\documentclass[twoside,reqno,11pt]{amsart}
\usepackage{latexsym}%%triangles:\lhd\rhd\unlhd\unrhd%
\newcommand{\qdn}{\hspace*{-1.5mm}}
\newcommand{\qqdn}{\hspace*{-2.5mm}}
\newcommand{\xqdn}{\hspace*{-5.0mm}}
\newcommand{\xxqdn}{\hspace*{-10mm}}

\newcommand{\fns}{\footnotesize}

\newcommand{\fnk}[3]{\left[\qdn\ba{#1}#2\\#3\ea\qdn\right]}
\newcommand{\ffnk}[4]{\left[\qdn\ba{#1}#3\\#4\ea{\!\Big|\:#2}\right]}

\newcommand{\be}{\begin{equation}}
\newcommand{\ee}{\end{equation}}
\newcommand{\ba}{\begin{array}}
\newcommand{\ea}{\end{array}}
\newcommand{\bmn}{\begin{eqnarray}}
\newcommand{\emn}{\end{eqnarray}}
\newcommand{\bnm}{\begin{eqnarray*}}
\newcommand{\enm}{\end{eqnarray*}}
\newcommand{\bln}{\begin{subequations}}
\newcommand{\eln}{\end{subequations}}

\newtheorem{thm}{Theorem}%[section]

\newtheorem{exam}{Example}

\newtheorem{entry}{Entry}%%%%%%%%%%%%%%%%

\newcommand{\bbtm}[4]{\bibitem{kn:#1}{#2,}~{#3,}~{#4.}}
\newcommand{\cito}[1]{\cite{kn:#1}}
\newcommand{\citu}[2]{\cite[#2]{kn:#1}}

\begin{document} %%%%%%%%%% This paper is published in %%%%%%%
{\fns% \today\hfill\copyright%% Printed in China} %%%%%%%%%%%%%%%
%%%%%%%%%%%%%%%%%%%%%%%%%%%%%%%%%%%%%%%%%%%%%%%%%%%%%%%%%%%%%%
\title{Generalizations of Andrews' curious identities}

\author{Chuanan Wei$^a$, Xiaoxia Wang$^b$}
\dedicatory{
$^A$Department of Information Technology\\
  Hainan Medical College, Haikou 571199, China\\
  $^B$Department of Mathematics\\
  Shanghai University, Shanghai 200444, China}
\thanks{\emph{Email addresses}:
      weichuanan78@163.com (C. Wei), xwang 913@126.com (X. Wang)}

\address{ }


\begin{abstract}
According to the method of series rearrangement, we establish two
generalizations of Andrews' curious $q$-series identity with an
extra integer parameter. The limiting cases of them produce two
extensions of Andrews' curious $_3F_2(\frac{3}{4})$-series identity
with an additional integer parameter. Meanwhile, several related
results are also given.
\end{abstract}

%%%%%%%%%%%%%%%%%%%%%%%%%%%%%%%%%%%%%%%%%%%%%%%%%%%%%%%%%%%%%%%%%%%
\maketitle\thispagestyle{empty}%%%%%%%%%%%%%%%%%%%%%%%%%%%%%%%%%%%%
\markboth{Chuanan Wei, Xiaoxia Wang}%%%%%%%%%%%%%%%%%%%%%%%%%%%%
         {Generalizations of Andrews' curious identities}

%%%%%%%%%%%%%%%%%%%%%%%%%%%%%%%%%%%%%%%%%%%%%%%%%%%%%%%%%%%%%%%%%%%
%%%%%%%%%%%%%%%%%%%%%%%%%%%%%%%%%%%%%%%%%%%%%%%%%%%%%%%%%%%%%%%%%%%
%%%%%%%%%%%%%%%%%%%%%%%%%%%%%%%%%%%%%%%%%%%%%%%%%%%%%%%%%%%%%%%%%%%
\section{Introduction}
%%%%%%%%%%%%%%%%%%%%%%%%%%%%%%%%%%%%%%%%%%%%%%%%%%%%%%
%%%%%%%%%%%%%%%%%%%%%%%%%%%%%%%%%%%%%%%%%%%%%%%%%%%%%%
For a complex variable $x$, define the shifted-factorial by
 \bnm
 (x)_n=
\begin{cases}
\prod_{k=0}^{n-1}(x+k),&n>0;\\
1,&n=0;\\
 \frac{(-1)^n}{\prod_{k=1}^{-n}(k-x)},&n<0.
\end{cases}
 \enm
 For simplifying the expressions, we shall use the symbol:
\[\qqdn\qdn\fnk{ccccc}{a,&b,&\cdots,&c}{\alpha,&\beta,&\cdots,&\gamma}_n
=\frac{(a)_n(b)_n\cdots(c)_n}{(\alpha)_n(\beta)_n\cdots(\gamma)_n}.\]
 Following Bailey~\cito{bailey}, define the hypergeometric series by
\[_{1+r}F_s\ffnk{cccc}{z}{a_0,&a_1,&\cdots,&a_r}{&b_1,&\cdots,&b_s}
 \:=\:\sum_{k=0}^\infty
\frac{(a_{0})_{k}(a_{1})_{k}\cdots(a_{r})_{k}}
 {k!(b_{1})_{k}\cdots(b_{s})_{k}}z^k,\]
where $\{a_{i}\}_{i\geq0}$ and $\{b_{j}\}_{j\geq1}$ are complex
parameters such that no zero factors appear in the denominators of
the summand on the right hand side. Then the curious
$_3F_2(\frac{3}{4})$-series identity due to Andrews
\citu{andrews-a}{Equation (22b)} reads as
 \bmn\label{andrews}
\qquad_3F_2\ffnk{cccc}{\frac{3}{4}}{-n,&a,&3a+n}{&\frac{3a}{2},&\frac{1+3a}{2}}
=\fnk{ccccc}{\frac{1}{3},&\frac{2}{3}}{\frac{1}{3}+a,&\frac{2}{3}+a}_{\frac{n}{3}}\chi(n),
 \emn
where the symbol $\chi(n)$ has been offered by
 \bnm
 \!\xxqdn\chi(n)=\begin{cases}
1,& n=3m;
\\[2mm]
0,&\text{otherwise}.
\end{cases}
\enm
 The reversal of it is the $_3F_2(\frac{4}{3})$-series identity:
 \bmn\label{andrews-reversal}
_3F_2\ffnk{cccc}{\frac{4}{3}}{-n,&\frac{3b}{2},&\frac{1+3b}{2}}{&3b,&b-\frac{n-2}{3}}
=\fnk{ccccc}{\frac{1}{3},&\frac{2}{3}}{\frac{1}{3}-b,&\frac{2}{3}+b}_{\frac{n}{3}}\chi(n).
 \emn
Some recent research work for hypergeometric series should be
mentioned. In accordance with inversion techniques, Chen and Chu
\cite{kn:chen-a,kn:chen-b} derived many $_3F_2(\frac{4}{3})$-series
identities related to \eqref{andrews-reversal}. According to series
rearrangement, Chu \cito{chu}, Lavoie \cito{lavoie-a} and Lavoie et.
al. \cito{lavoie-b} explored $_3F_2(1)$-series identities.

 For two complex numbers $x$ and
$q$, define the $q$-shifted factorial by
  \bnm
 (x;q)_n=
\begin{cases}
\prod_{i=0}^{n-1}(1-xq^i),&n>0;\\
\quad1,&n=0;\\
\frac{1}{\prod_{j=n}^{-1}(1-xq^j)},&n<0.
\end{cases}
\enm
 The fraction form of it reads as
\[\qqdn\qdn\ffnk{ccccc}{q}{a,&b,&\cdots,&c}{\alpha,&\beta,&\cdots,&\gamma}_n
=\frac{(a;q)_n(b;q)_n\cdots(c;q)_n}{(\alpha;q)_n(\beta;q)_n\cdots(\gamma;q)_n}.\]
 Following Gasper and Rahman \cito{gasper}, define the $q$-series by
\[_{1+r}\phi_s\ffnk{cccc}{q;z}{a_0,&a_1,&\cdots,&a_r}
{&b_1,&\cdots,&b_s}
 =\sum_{k=0}^\infty
\ffnk{ccccc}{q}{a_0,&a_1,&\cdots,&a_r}{q,&b_1,&\cdots,&b_s}_kz^k,\]
where $\{a_i\}_{i\geq0}$ and $\{b_j\}_{j\geq1}$ are complex
parameters such that no zero factors appear in the denominators of
the summand on the right hand side. Then the terminating
$_6\phi_5$-series identity (cf. Gasper and Rahman \citu{gasper}{p.
42}) can be expressed as
 \bmn\label{terminating-65}
\:\:\qquad {_6\phi_5}\ffnk{cccccccc}{q;\frac{q^{1+\ell}a}{bc}}
{a,\:q\sqrt{a},\:-q\sqrt{a},\:b,\:c,\:q^{-\ell}}
 {\sqrt{a},\:-\sqrt{a},\:qa/b,\:qa/c,\:q^{1+\ell}a}
 =\ffnk{ccccc}{q}{qa,qa/bc}{qa/b,qa/c}_\ell.
 \emn
Recall a curious $q$-series identity due to Andrews
\citu{andrews-b}{Equation (4.7)}:
  \bmn\qquad
\label{q-andrews}
\sum_{k=0}^n(a;q^3)_k\ffnk{ccccc}{q}{q^{-n},q^{n}a}
{q,\sqrt{a},-\sqrt{a},\sqrt{qa},-\sqrt{qa}}_kq^k=a^{n/3}\ffnk{ccccc}{q^3}{q,q^2}{qa,q^2a}_{\frac{n}{3}}\chi(n).
 \emn
The reversal of it can be stated as
 \bmn\qquad\quad\:
\label{q-andrews-reversal}
\sum_{k=0}^n\frac{1}{(q^{2-n}b;q^3)_k}\ffnk{ccccc}{q}{q^{-n},\sqrt{b},-\sqrt{b},\sqrt{qb},-\sqrt{qb}}
{q,b}_kq^k=\ffnk{ccccc}{q^3}{q,q^2}{q/b,q^2b}_{\frac{n}{3}}\chi(n).
 \emn

 Inspired by the work of \cito{chen-a}-\cito{lavoie-b}, we shall establish
 two generalizations of \eqref{q-andrews}, which involve two extensions of
 \eqref{andrews}, by means of series rearrangement in section 2. The
 reversal of them creates two generalizations of \eqref{q-andrews-reversal}, which involve two extensions of
 \eqref{andrews-reversal}, in section 3.

%%%%%%%%%%%%%%%%%%%%%%%%%%%%%%%%%%%%%%%%%%%%%%%%%%%%%%
%%%%%%%%%%%%%%%%%%%%%%%%%%%%%%%%%%%%%%%%%%%%%%%%%%%%%%
\section{Generalizations of Andrews' curious identities}
%%%%%%%%%%%%%%%%%%%%%%%%%%%%%%%%%%%%%%%%%%%%%%%%%%%%%%
%%%%%%%%%%%%%%%%%%%%%%%%%%%%%%%%%%%%%%%%%%%%%%%%%%%%%%

\begin{thm}\label{thm-a}
For a nonnegative integer $\ell$ and a complex number $a$, there
holds
 \bnm
&&\xqdn\sum_{k=0}^n(a;q^3)_k\ffnk{ccccc}{q}{q^{-n},q^{n-\ell}a}
{q,\sqrt{a},-\sqrt{a},\sqrt{qa},-\sqrt{qa}}_kq^k=\ffnk{ccccc}{q}{q^{n-\ell}a}{q^{2n-\ell}a}_{\ell}
\\&&\xqdn\:\:\times\:\sum_{i=0}^{\ell}a^{\frac{n+2i}{3}}
q^{(\ell+2n-i)i}\frac{a-q^{2i-2n}}{a-q^{-2n}}
\ffnk{ccccc}{q}{q^{-\ell},q^{-n},q^{-2n}/a}
{q,q^{1+\ell-2n}/a,q^{n-i}a}_i
\\&&\xqdn\:\:\times\:
\ffnk{ccccc}{q^3}{q,q^2}{qa,q^2a}_{\frac{n-i}{3}}\chi(n-i).
 \enm
\end{thm}

\begin{proof}
 Letting $a\to q^{-2n}/a$, $b\to q^{k-n}$ and $c\to0$ for
\eqref{terminating-65}, we get the equation:
 \bnm\qquad\:\:
\ffnk{ccccc}{q}{q^{n-\ell+k}a}{q^{2n-\ell}a}_{\ell}\sum_{i=0}^{\ell}a^i
q^{(\ell+2n-i)i}\frac{a-q^{2i-2n}}{a-q^{-2n}}
\ffnk{ccccc}{q}{q^{-\ell},q^{k-n},q^{-2n}/a}
{q,q^{1+\ell-2n}/a,q^{n+k-i}a}_i=1.
 \enm
 Then there is the following relation:
  \bnm
&&\xqdn\sum_{k=0}^n(a;q^3)_k\ffnk{ccccc}{q}{q^{-n},q^{n-\ell}a}
{q,\sqrt{a},-\sqrt{a},\sqrt{qa},-\sqrt{qa}}_kq^k
\\&&\xqdn\:\:=\:\sum_{k=0}^n(a;q^3)_k\ffnk{ccccc}{q}{q^{-n},q^{n-\ell}a}
{q,\sqrt{a},-\sqrt{a},\sqrt{qa},-\sqrt{qa}}_kq^k
\\&&\xqdn\:\:\times\:
 \ffnk{ccccc}{q}{q^{n-\ell+k}a}{q^{2n-\ell}a}_{\ell}\sum_{i=0}^{\ell}a^i
q^{(\ell+2n-i)i}\frac{a-q^{2i-2n}}{a-q^{-2n}}
\ffnk{ccccc}{q}{q^{-\ell},q^{k-n},q^{-2n}/a}
{q,q^{1+\ell-2n}/a,q^{n+k-i}a}_i.
 \enm
Interchange the summation order for the last double sum to achieve
  \bnm
&&\xxqdn\sum_{k=0}^n(a;q^3)_k\ffnk{ccccc}{q}{q^{-n},q^{n-\ell}a}
{q,\sqrt{a},-\sqrt{a},\sqrt{qa},-\sqrt{qa}}_kq^k
\\&&\xxqdn\:\:=\:\ffnk{ccccc}{q}{q^{n-\ell}a}{q^{2n-\ell}a}_{\ell}
\sum_{i=0}^{\ell}a^i q^{(\ell+2n-i)i}\frac{a-q^{2i-2n}}{a-q^{-2n}}
\ffnk{ccccc}{q}{q^{-\ell},q^{-n},q^{-2n}/a}
{q,q^{1+\ell-2n}/a,q^{n-i}a}_i
\\&&\xxqdn\:\:\times\:
\sum_{k=0}^{n-i}(a;q^3)_k\ffnk{ccccc}{q}{q^{i-n},q^{n-i}a}
{q,\sqrt{a},-\sqrt{a},\sqrt{qa},-\sqrt{qa}}_kq^k.
 \enm
 Calculating the series on the last line by
\eqref{q-andrews}, we attain Theorem \ref{thm-a} to complete the
proof.
\end{proof}

When $\ell=0$, Theorem \ref{thm-a} reduces to \eqref{q-andrews}
exactly. Other two examples are displayed as follows.

\begin{exam}[$\ell=1$ in Theorem \ref{thm-a}]
\bnm
 &&\xxqdn\xxqdn\sum_{k=0}^n(a;q^3)_k\ffnk{ccccc}{q}{q^{-n},q^{n-1}a}
{q,\sqrt{a},-\sqrt{a},\sqrt{qa},-\sqrt{qa}}_kq^k
 \\&&\xxqdn\xxqdn\:=\:
\begin{cases}
\frac{a^m(1-aq^{-1})}{1-aq^{6m-1}}
\ffnk{ccccc}{q^3}{q,q^2}{qa,q^{-1}a}_m,&\qqdn n=3m;
\\[4mm]
\frac{a^{m+1}q^{3m}(1-q)}{1-aq^{6m+1}}
\ffnk{ccccc}{q^3}{q^2,q^4}{qa,q^2a}_m,&\qqdn n=1+3m;
\\[4mm]
0,&\qqdn n=2+3m.
\end{cases}
\enm
\end{exam}

\begin{exam}[$\ell=2$ in Theorem \ref{thm-a}]
\bnm
 &&\sum_{k=0}^n(a;q^3)_k\ffnk{ccccc}{q}{q^{-n},q^{n-2}a}
{q,\sqrt{a},-\sqrt{a},\sqrt{qa},-\sqrt{qa}}_kq^k
 \\&&\:=\:
\begin{cases}
\frac{a^m(1-aq^{-2})(1-aq^{-1})}{(1-aq^{6m-2})(1-aq^{6m-1})}
\ffnk{ccccc}{q^3}{q,q^2}{q^{-1}a,q^{-2}a}_m,&\qqdn n=3m;
\\[4mm]
\frac{a^{m+1}q^{3m-1}(1-q^2)(1-aq^{-1})}{(1-aq^{6m-1})(1-aq^{6m+1})}
\ffnk{ccccc}{q^3}{q^2,q^4}{q^{-1}a,qa}_m,&\qqdn n=1+3m;
\\[4mm]
\frac{a^{m+2}q^{6m}(1-q)(1-q^2)}{(1-aq^{6m+1})(1-aq^{6m+2})}
\ffnk{ccccc}{q^3}{q^4,q^5}{qa,q^2a}_m,&\qqdn n=2+3m.
\end{cases}
\enm
\end{exam}

Performing the replacement $a\to q^{3a}$ in Theorem \ref{thm-a} and
then letting $q\to1$, we obtain the following equation.

\begin{thm}\label{thm-b}
For a nonnegative integer $\ell$ and a complex number $a$, there
holds
 \bnm
&&\:\xxqdn_3F_2\ffnk{cccc}{\frac{3}{4}}{-n,&a,&3a+n-\ell}{&\frac{3a}{2},&\frac{1+3a}{2}}
=\fnk{ccccc}{1-n-3a}{1-2n-3a}_{\ell}
\\&&\:\xxqdn\:\:\times\:\:\sum_{i=0}^{\ell}(-1)^i\frac{3a+2n-2i}{3a+2n}\fnk{ccccc}{-\ell,-n,-2n-3a}{1,1-n-3a,1+\ell-2n-3a}_i
\\&&\:\xxqdn\:\:\times\:\:\fnk{ccccc}{\frac{1}{3},&\frac{2}{3}}{\frac{1}{3}+a,&\frac{2}{3}+a}_{\frac{n-i}{3}}\chi(n-i).
 \enm
\end{thm}

When $\ell=0$, Theorem \ref{thm-b} reduces to \eqref{andrews}
exactly. Other two examples are laid out as follows.

\begin{exam}[$\ell=1$ in Theorem \ref{thm-b}]
\bnm
 &&\xxqdn\xxqdn_3F_2\ffnk{cccc}{\frac{3}{4}}{-n,&a,&3a+n-1}{&\frac{3a}{2},&\frac{1+3a}{2}}
 \\&&\xxqdn\xxqdn\:=\:
\begin{cases}
\frac{3a-1}{3a+6m-1}\fnk{ccccc}{\frac{1}{3},&\frac{2}{3}}{-\frac{1}{3}+a,&\frac{1}{3}+a}_m,&\qqdn
n=3m;
\\[4mm]
\frac{1}{3a+6m+1}\fnk{ccccc}{\frac{2}{3},&\frac{4}{3}}{\frac{1}{3}+a,&\frac{2}{3}+a}_m,&\qqdn
n=1+3m;
\\[4mm]
0,&\qqdn n=2+3m.
\end{cases}
\enm
\end{exam}

\begin{exam}[$\ell=2$ in Theorem \ref{thm-b}]
\bnm
 &&_3F_2\ffnk{cccc}{\frac{3}{4}}{-n,&a,&3a+n-2}{&\frac{3a}{2},&\frac{1+3a}{2}}
 \\&&\:=\:
\begin{cases}
\frac{(3a-2)(3a-1)}{(3a+6m-2)(3a+6m-1)}\fnk{ccccc}{\frac{1}{3},&\frac{2}{3}}{-\frac{1}{3}+a,&-\frac{2}{3}+a}_m,&\qqdn
n=3m;
\\[4mm]
\frac{2(3a-1)}{(3a+6m-1)(3a+6m+1)}\fnk{ccccc}{\frac{2}{3},&\frac{4}{3}}{-\frac{1}{3}+a,&\frac{1}{3}+a}_m,&\qqdn
n=1+3m;
\\[4mm]
\frac{2}{(3a+6m+1)(3a+6m+2)}\fnk{ccccc}{\frac{4}{3},&\frac{5}{3}}{\frac{1}{3}+a,&\frac{2}{3}+a}_m,&\qqdn
n=2+3m.
\end{cases}
\enm
\end{exam}

\begin{thm}\label{thm-c}
For a nonnegative integer $\ell$ and a complex number $a$, there
holds
 \bnm
&&\xqdn\sum_{k=0}^n(a;q^3)_k\ffnk{ccccc}{q}{q^{-n},q^{n-\ell}a}
{q,\sqrt{a},-\sqrt{a},\sqrt{q^{1-2\ell}a},-\sqrt{qa}}_kq^k
 =\ffnk{ccccc}{q}{q^{n-\ell}a,q^{\frac{1}{2}+n-\ell}\!\sqrt{a}}{q^{2n-\ell}a,q^{\frac{1}{2}-\ell}\!\sqrt{a}}_{\ell}
\\&&\xqdn\:\:\times\:\sum_{i=0}^{\ell}a^{\frac{2n-5i}{6}}
q^{(\frac{5}{2}+\ell)i}\frac{1-q^{2n-2i}a}{1-q^{2n}a}
\ffnk{ccccc}{q}{q^{-\ell},q^{-n},q^{-2n}/a}
{q,q^{1+\ell-2n}/a,q^{1-n}/a}_i
\\&&\xqdn\:\:\times\:
\ffnk{ccccc}{q^3}{q,q^2}{qa,q^2a}_{\frac{n-i}{3}}\chi(n-i).
 \enm
\end{thm}

\begin{proof}
 Letting $a\to q^{-2n}/a$, $b\to q^{k-n}$ and $c\to q^{\frac{1}{2}-n}/\sqrt{a}$ for
\eqref{terminating-65}, the resulting equation reads as
 \bnm
&&\ffnk{ccccc}{q}{q^{n-\ell}a,q^{\frac{1}{2}+n-\ell}\!\sqrt{a}}{q^{2n-\ell}a,q^{\frac{1}{2}+k-\ell}\!\sqrt{a}}_{\ell}
\sum_{i=0}^{\ell}\frac{q^{(5/2+\ell)i}}{a^{i/2}}\frac{1-q^{2n-2i}a}{1-q^{2n}a}
\ffnk{ccccc}{q}{q^{-\ell},q^{k-n},q^{-2n}/a}
{q,q^{1+\ell-2n}/a,q^{1-n}/a}_i
\\&&\times\:\ffnk{ccccc}{q}{q^{n-\ell+k}a}{q^{n-\ell}a}_{\ell-i}=1.
 \enm
 Then we can proceed as follows:
  \bnm
&&\xqdn\sum_{k=0}^n(a;q^3)_k\ffnk{ccccc}{q}{q^{-n},q^{n-\ell}a}
{q,\sqrt{a},-\sqrt{a},\sqrt{q^{1-2\ell}a},-\sqrt{qa}}_kq^k
\\&&\xqdn\:\:=\:\sum_{k=0}^n(a;q^3)_k\ffnk{ccccc}{q}{q^{-n},q^{n-\ell}a}
{q,\sqrt{a},-\sqrt{a},\sqrt{q^{1-2\ell}a},-\sqrt{qa}}_kq^k
\\&&\xqdn\:\:\times\:
\ffnk{ccccc}{q}{q^{n-\ell}a,q^{\frac{1}{2}+n-\ell}\!\sqrt{a}}{q^{2n-\ell}a,q^{\frac{1}{2}+k-\ell}\!\sqrt{a}}_{\ell}
\sum_{i=0}^{\ell}\frac{q^{(5/2+\ell)i}}{a^{i/2}}\frac{1-q^{2n-2i}a}{1-q^{2n}a}
\ffnk{ccccc}{q}{q^{-\ell},q^{k-n},q^{-2n}/a}
{q,q^{1+\ell-2n}/a,q^{1-n}/a}_i
\\&&\xqdn\:\:\times\:
\ffnk{ccccc}{q}{q^{n-\ell+k}a}{q^{n-\ell}a}_{\ell-i}.
 \enm
Interchange the summation order for the last double sum to get
 \bnm
&&\xxqdn\sum_{k=0}^n(a;q^3)_k\ffnk{ccccc}{q}{q^{-n},q^{n-\ell}a}
{q,\sqrt{a},-\sqrt{a},\sqrt{q^{1-2\ell}a},-\sqrt{qa}}_kq^k
 =\ffnk{ccccc}{q}{q^{n-\ell}a,q^{\frac{1}{2}+n-\ell}\!\sqrt{a}}{q^{2n-\ell}a,q^{\frac{1}{2}-\ell}\!\sqrt{a}}_{\ell}
\\&&\xxqdn\:\:\times\:\sum_{i=0}^{\ell}\frac{q^{(5/2+\ell)i}}{a^{i/2}}\frac{1-q^{2n-2i}a}{1-q^{2n}a}
\ffnk{ccccc}{q}{q^{-\ell},q^{-n},q^{-2n}/a}
{q,q^{1+\ell-2n}/a,q^{1-n}/a}_i
\\&&\xxqdn\:\:\times\:
\sum_{k=0}^{n-i}(a;q^3)_k\ffnk{ccccc}{q}{q^{i-n},q^{n-i}a}
{q,\sqrt{a},-\sqrt{a},\sqrt{qa},-\sqrt{qa}}_kq^k.
 \enm
 Evaluating the series on the last line by
\eqref{q-andrews}, we achieve Theorem \ref{thm-c} to complete the
proof.
\end{proof}

When $\ell=0$, Theorem \ref{thm-c} also reduces to \eqref{q-andrews}
exactly. Other two examples are displayed as follows.

\begin{exam}[$\ell=1$ in Theorem \ref{thm-c}]
\bnm
 &&\qqdn\xqdn\xxqdn\xxqdn\sum_{k=0}^n(a;q^3)_k\ffnk{ccccc}{q}{q^{-n},q^{n-1}a}
{q,\sqrt{a},-\sqrt{a},\sqrt{q^{-1}a},-\sqrt{qa}}_kq^k
 \\&&\qqdn\xqdn\xxqdn\xxqdn\:=\:
\begin{cases}
\frac{a^m(\sqrt{q}+\sqrt{a})}{\sqrt{q}+q^{3m}\!\sqrt{a}}
\ffnk{ccccc}{q^3}{q,q^2}{qa,q^{-1}a}_m,&\qqdn n=3m;
\\[4mm]
\frac{a^{m+1/2}(q-1)}{(\sqrt{q}-\sqrt{a})(1+q^{3m+1/2}\!\sqrt{a})}
\ffnk{ccccc}{q^3}{q^2,q^4}{qa,q^2a}_m,&\qqdn n=1+3m;
\\[4mm]
0,&\qqdn n=2+3m.
\end{cases}
\enm
\end{exam}

\begin{exam}[$\ell=2$ in Theorem \ref{thm-c}]
\bnm
 &&\sum_{k=0}^n(a;q^3)_k\ffnk{ccccc}{q}{q^{-n},q^{n-2}a}
{q,\sqrt{a},-\sqrt{a},\sqrt{q^{-3}a},-\sqrt{qa}}_kq^k
 \\&&\:=\:
\begin{cases}
\frac{a^m(q^2-a)(\sqrt{q}+\sqrt{a})(1-q^{3m-3/2}\!\sqrt{a})}{(q^2-q^{6m}a)(1-q^{-3/2}\!\sqrt{a})(\sqrt{q}+q^{3m}\!\sqrt{a})}
\ffnk{ccccc}{q^3}{q,q^2}{q^{-1}a,q^{-2}a}_m,&\qqdn n=3m;
\\[4mm]
\frac{a^{m+1/2}(q^2-1)(\sqrt{q}+\sqrt{a})}{(q^2-\sqrt{qa})(1+q^{3m-1/2}\!\sqrt{a})(1+q^{3m+1/2}\!\sqrt{a})}
\ffnk{ccccc}{q^3}{q^2,q^4}{q^{-1}a,qa}_m,&\qqdn n=1+3m;
\\[4mm]
\frac{a^{m+1}(1-q)(1-q^2)(1-q^{3m+3/2}\!\sqrt{a})}{(\sqrt{q}-\sqrt{a})(\sqrt{q^3}-\sqrt{a})(1+q^{3m+1/2}\!\sqrt{a})(1-q^{6m+2}a)}
\ffnk{ccccc}{q^3}{q^4,q^5}{qa,q^2a}_m,&\qqdn n=2+3m.
\end{cases}
\enm
\end{exam}

Employing the substitution $a\to q^{3a}$ in Theorem \ref{thm-c} and
then letting $q\to1$, we obtain the following equation.

\begin{thm}\label{thm-d}
For a nonnegative integer $\ell$ and a complex number $a$, there
holds
 \bnm
&&_3F_2\ffnk{cccc}{\frac{3}{4}}{-n,&a,&3a+n-\ell}{&\frac{3a}{2},&\frac{1+3a}{2}-\ell}
=\fnk{ccccc}{1-3a-n,\frac{1-3a}{2}-n}{1-3a-2n,\frac{1-3a}{2}}_{\ell}
\\&&\:\:\times\:\:\sum_{i=0}^{\ell}\frac{3a+2n-2i}{3a+2n}\fnk{ccccc}{-\ell,-n,-2n-3a}{1,1-n-3a,1+\ell-2n-3a}_i
\\&&\:\:\times\:\:\fnk{ccccc}{\frac{1}{3},&\frac{2}{3}}{\frac{1}{3}+a,&\frac{2}{3}+a}_{\frac{n-i}{3}}\chi(n-i).
 \enm
\end{thm}

When $\ell=0$, Theorem \ref{thm-d} also reduces to \eqref{andrews}
exactly. Other two examples are laid out as follows.

\begin{exam}[$\ell=1$ in Theorem \ref{thm-d}]
\bnm
 &&\qdn\xxqdn\xxqdn_3F_2\ffnk{cccc}{\frac{3}{4}}{-n,&a,&3a+n-1}{&\frac{3a-1}{2},&\frac{3a}{2}}
 \\&&\qdn\xxqdn\xxqdn\:=\:
\begin{cases}
\fnk{ccccc}{\frac{1}{3},&\frac{2}{3}}{-\frac{1}{3}+a,&\frac{1}{3}+a}_m,&\qqdn
n=3m;
\\[4mm]
\frac{1}{1-3a}\fnk{ccccc}{\frac{2}{3},&\frac{4}{3}}{\frac{1}{3}+a,&\frac{2}{3}+a}_m,&\qqdn
n=1+3m;
\\[4mm]
0,&\qqdn n=2+3m.
\end{cases}
\enm
\end{exam}

\begin{exam}[$\ell=2$ in Theorem \ref{thm-d}]
\bnm
 &&_3F_2\ffnk{cccc}{\frac{3}{4}}{-n,&a,&3a+n-2}{&\frac{3a-3}{2},&\frac{3a}{2}}
 \\&&\:=\:
\begin{cases}
\frac{(3a-2)(a+2m-1)}{(a-1)(3a+6m-2)}\fnk{ccccc}{\frac{1}{3},&\frac{2}{3}}{-\frac{1}{3}+a,&-\frac{2}{3}+a}_m,&\qqdn
n=3m;
\\[4mm]
\frac{2}{3(1-a)}\fnk{ccccc}{\frac{2}{3},&\frac{4}{3}}{-\frac{1}{3}+a,&\frac{1}{3}+a}_m,&\qqdn
n=1+3m;
\\[4mm]
\frac{2(a+2m+1)}{(a-1)(3a-1)(3a+6m+2)}\fnk{ccccc}{\frac{4}{3},&\frac{5}{3}}{\frac{1}{3}+a,&\frac{2}{3}+a}_m,&\qqdn
n=2+3m.
\end{cases}
\enm
\end{exam}

%%%%%%%%%%%%%%%%%%%%%%%%%%%%%%%%%%%%%%%%%%%%%%%%%%%%%%
%%%%%%%%%%%%%%%%%%%%%%%%%%%%%%%%%%%%%%%%%%%%%%%%%%%%%%
\section{Several identities from reversal}
%%%%%%%%%%%%%%%%%%%%%%%%%%%%%%%%%%%%%%%%%%%%%%%%%%%%%%
%%%%%%%%%%%%%%%%%%%%%%%%%%%%%%%%%%%%%%%%%%%%%%%%%%%%%%
Performing the replacement $k\to n-k$ and $a\to q^{1-2n}/b$ in
Theorem \ref{thm-a}, we derive the following equation.

\begin{thm}\label{thm-e}
For a nonnegative integer $\ell$ and a complex number $b$, there
holds
 \bnm
&&\xqdn\sum_{k=0}^n\frac{1}{(q^{2-n}b;q^3)_k}\ffnk{ccccc}{q}{q^{-n},\sqrt{b},-\sqrt{b},\sqrt{qb},-\sqrt{qb}}
{q,q^{\ell}b}_kq^{(1+\ell)k}=\frac{(q^{1-2n}/b;q)_n}{(q^{1-2n}/b;q^3)_n}
\\&&\xqdn\:\:\times\:\sum_{i=0}^{\ell}(-1)^n\frac{q^{\{(4+4n+6\ell-6i)i-n-n^2\}/6}}{b^{(n+2i)/3}}
\frac{q-q^{2i}b}{q-b}
\ffnk{ccccc}{q}{q^{-\ell},q^{-n},b/q}{q,q^{\ell}b,q^{1-n-i}/b}_i
\\&&\xqdn\:\:\times\:
\ffnk{ccccc}{q^3}{q,q^2}{q^{2-2n}/b,q^{3-2n}/b}_{\frac{n-i}{3}}\chi(n-i).
 \enm
\end{thm}

When $\ell=0$, Theorem \ref{thm-e} reduces to
\eqref{q-andrews-reversal} exactly. Other two examples are displayed
as follows.

\begin{exam}[$\ell=1$ in Theorem \ref{thm-e}]
\bnm
 &&\xxqdn\xxqdn\sum_{k=0}^n\frac{1}{(q^{2-n}b;q^3)_k}\ffnk{ccccc}{q}{q^{-n},\sqrt{b},-\sqrt{b},\sqrt{qb},-\sqrt{qb}}
{q,qb}_kq^{2k}
 \\&&\xxqdn\xxqdn\:=\:
\begin{cases}
\ffnk{ccccc}{q^3}{q,q^2}{q/b,q^{2}b}_m,&\qqdn n=3m;
\\[4mm]
\frac{1-q}{1-qb} \ffnk{ccccc}{q^3}{q^2,q^4}{q^2/b,q^4b}_m,&\qqdn
n=1+3m;
\\[4mm]
0,&\qqdn n=2+3m.
\end{cases}
\enm
\end{exam}

\begin{exam}[$\ell=2$ in Theorem \ref{thm-e}]
\bnm
 &&\xxqdn\xxqdn\sum_{k=0}^n\frac{1}{(q^{2-n}b;q^3)_k}\ffnk{ccccc}{q}{q^{-n},\sqrt{b},-\sqrt{b},\sqrt{qb},-\sqrt{qb}}
{q,q^2b}_kq^{3k}
 \\&&\xxqdn\xxqdn\:=\:
\begin{cases}
\ffnk{ccccc}{q^3}{q,q^2}{q/b,q^{2}b}_m,&\qqdn n=3m;
\\[4mm]
\frac{1-q^2}{1-q^2b} \ffnk{ccccc}{q^3}{q^2,q^4}{q^2/b,q^4b}_m,&\qqdn
n=1+3m;
\\[4mm]
\frac{(1-q)(1-q^2)}{(1-q^2b)(1-q^3b)}
\ffnk{ccccc}{q^3}{q^4,q^5}{q^3/b,q^6b}_m,&\qqdn n=2+3m.
\end{cases}
\enm
\end{exam}

Employing the substitution $a\to q^{3a}$ in Theorem \ref{thm-e} and
then letting $q\to1$, we get the following equation.

\begin{thm}\label{thm-f}
For a nonnegative integer $\ell$ and a complex number $a$, there
holds
 \bnm
&&\xxqdn_3F_2\ffnk{cccc}{\frac{4}{3}}{-n,&\frac{3b}{2},&\frac{1+3b}{2}}{&3b+\ell,&b-\frac{n-2}{3}}
=\fnk{ccccc}{3b+n}{\frac{1-2n}{3}-b}_{n}
\\&&\xxqdn\:\:\times\:\:\sum_{i=0}^{\ell}\frac{(-1)^i}{3^n}\frac{3b-1+2i}{3b-1}\fnk{ccccc}{-\ell,-n,3b-1}{1,3b+n,3b+\ell}_i
\\&&\xxqdn\:\:\times\:\:\fnk{ccccc}{\frac{1}{3},&\frac{2}{3}}{\frac{2-2n-3b}{3},&\frac{3-2n-3b}{3}}_{\frac{n-i}{3}}\chi(n-i).
 \enm
\end{thm}

When $\ell=0$, Theorem \ref{thm-e} reduces to
\eqref{andrews-reversal} exactly. Other two examples are laid out as
follows.

\begin{exam}[$\ell=1$ in Theorem \ref{thm-e}: Chen and Chu \citu{chen-b}{Theorem 4}]
\bnm
 &&\xxqdn\xxqdn_3F_2\ffnk{cccc}{\frac{4}{3}}{-n,&\frac{3b}{2},&\frac{1+3b}{2}}{&1+3b,&b-\frac{n-2}{3}}
 \\&&\xxqdn\xxqdn\:=\:
\begin{cases}
\fnk{ccccc}{\frac{1}{3},&\frac{2}{3}}{\frac{1}{3}-b,&\frac{2}{3}+b}_m,&\qqdn
n=3m;
\\[4mm]
\frac{1}{3b+1}\fnk{ccccc}{\frac{2}{3},&\frac{4}{3}}{\frac{2}{3}-b,&\frac{4}{3}+b}_m,&\qqdn
n=1+3m;
\\[4mm]
0,&\qqdn n=2+3m.
\end{cases}
\enm
\end{exam}

\begin{exam}[$\ell=2$ in Theorem \ref{thm-e}: Chen and Chu \citu{chen-b}{Theorem 25}]
\bnm
 &&\xxqdn_3F_2\ffnk{cccc}{\frac{4}{3}}{-n,&\frac{3b}{2},&\frac{1+3b}{2}}{&2+3b,&b-\frac{n-2}{3}}
 \\&&\xxqdn\:=\:
\begin{cases}
\fnk{ccccc}{\frac{1}{3},&\frac{2}{3}}{\frac{1}{3}-b,&\frac{2}{3}+b}_m,&\qqdn
n=3m;
\\[4mm]
\frac{2}{3b+2}\fnk{ccccc}{\frac{2}{3},&\frac{4}{3}}{\frac{2}{3}-b,&\frac{4}{3}+b}_m,&\qqdn
n=1+3m;
\\[4mm]
\frac{2}{3(b+1)(3b+2)}\fnk{ccccc}{\frac{4}{3},&\frac{5}{3}}{1-b,&2+b}_m,&\qqdn
n=2+3m.
\end{cases}
\enm
\end{exam}

Performing the replacement $k\to n-k$ and $a\to q^{1-2n}/b$ in
Theorem \ref{thm-c}, we deduce the following equation.

\begin{thm}\label{thm-g}
For a nonnegative integer $\ell$ and a complex number $b$, there
holds
 \bnm
&&\xqdn\sum_{k=0}^n\frac{1}{(q^{2-n}b;q^3)_k}\ffnk{ccccc}{q}{q^{-n},q^{\ell}\!\sqrt{b},-\sqrt{b},\sqrt{qb},-\sqrt{qb}}
{q,q^{\ell}b}_kq^k=\frac{(q^{1-2n}/b;q)_n}{(q^{1-2n}/b;q^3)_n}
\\&&\xqdn\:\:\times\:\sum_{i=0}^{\ell}(-1)^n\frac{q^{\{(10n+6\ell-2)i-n-n^2\}/6}}{b^{(2n-5i)/3}}
\frac{q-q^{2i}b}{q-b}
\ffnk{ccccc}{q}{q^{-\ell},q^{-n},b/q}{q,q^{\ell}b,q^{n}b}_i
\\&&\xqdn\:\:\times\:
\ffnk{ccccc}{q^3}{q,q^2}{q^{2-2n}/b,q^{3-2n}/b}_{\frac{n-i}{3}}\chi(n-i).
 \enm
\end{thm}

When $\ell=0$, Theorem \ref{thm-g} reduces to
\eqref{q-andrews-reversal} exactly. Other two examples are displayed
as follows.

\begin{exam}[$\ell=1$ in Theorem \ref{thm-g}]
\bnm
 &&\xxqdn\xxqdn\sum_{k=0}^n\frac{1}{(q^{2-n}b;q^3)_k}\ffnk{ccccc}{q}{q^{-n},q\sqrt{b},-\sqrt{b},\sqrt{qb},-\sqrt{qb}}
{q,qb}_kq^{k}
 \\&&\xxqdn\xxqdn\:=\:
\begin{cases}
\ffnk{ccccc}{q^3}{q,q^2}{q/b,q^{2}b}_m,&\qqdn n=3m;
\\[4mm]
\frac{\sqrt{b}(q-1)}{1-qb}
\ffnk{ccccc}{q^3}{q^2,q^4}{q^2/b,q^4b}_m,&\qqdn n=1+3m;
\\[4mm]
0,&\qqdn n=2+3m.
\end{cases}
\enm
\end{exam}

\begin{exam}[$\ell=2$ in Theorem \ref{thm-g}]
\bnm
 &&\xxqdn\xxqdn\sum_{k=0}^n\frac{1}{(q^{2-n}b;q^3)_k}\ffnk{ccccc}{q}{q^{-n},q^2\!\sqrt{b},-\sqrt{b},\sqrt{qb},-\sqrt{qb}}
{q,q^2b}_kq^{k}
 \\&&\xxqdn\xxqdn\:=\:
\begin{cases}
\ffnk{ccccc}{q^3}{q,q^2}{q/b,q^{2}b}_m,&\qqdn n=3m;
\\[4mm]
\frac{\sqrt{b}(q^2-1)}{1-q^2b}
\ffnk{ccccc}{q^3}{q^2,q^4}{q^2/b,q^4b}_m,&\qqdn n=1+3m;
\\[4mm]
\frac{qb(1-q)(1-q^2)}{(1-q^2b)(1-q^3b)}
\ffnk{ccccc}{q^3}{q^4,q^5}{q^3/b,q^6b}_m,&\qqdn n=2+3m.
\end{cases}
\enm
\end{exam}

Employing the substitution $a\to q^{3a}$ in Theorem \ref{thm-g} and
then letting $q\to1$, we obtain the following equation.

\begin{thm}\label{thm-h}
For a nonnegative integer $\ell$ and a complex number $a$, there
holds
 \bnm
&&\qdn\xqdn_3F_2\ffnk{cccc}{\frac{4}{3}}{-n,&\frac{3b}{2}+\ell,&\frac{1+3b}{2}}{&3b+\ell,&b-\frac{n-2}{3}}
=\frac{1}{3^n}\fnk{ccccc}{3b+n}{\frac{1-2n}{3}-b}_{n}
\\&&\qdn\xqdn\:\:\times\:\:\sum_{i=0}^{\ell}\frac{3b-1+2i}{3b-1}\fnk{ccccc}{-\ell,-n,3b-1}{1,3b+n,3b+\ell}_i
\\&&\qdn\xqdn\:\:\times\:\:\fnk{ccccc}{\frac{1}{3},&\frac{2}{3}}{\frac{2-2n-3b}{3},&\frac{3-2n-3b}{3}}_{\frac{n-i}{3}}\chi(n-i).
 \enm
\end{thm}

When $\ell=0$, Theorem \ref{thm-h} reduces to
\eqref{andrews-reversal} exactly. Other two examples are laid out as
follows.

\begin{exam}[$\ell=1$ in Theorem \ref{thm-h}]
\bnm
 &&\xxqdn\xxqdn_3F_2\ffnk{cccc}{\frac{4}{3}}{-n,&\frac{1+3b}{2},&\frac{2+3b}{2}}{&1+3b,&b-\frac{n-2}{3}}
 \\&&\xxqdn\xxqdn\:=\:
\begin{cases}
\fnk{ccccc}{\frac{1}{3},&\frac{2}{3}}{\frac{1}{3}-b,&\frac{2}{3}+b}_m,&\qqdn
n=3m;
\\[4mm]
\frac{-1}{3b+1}\fnk{ccccc}{\frac{2}{3},&\frac{4}{3}}{\frac{2}{3}-b,&\frac{4}{3}+b}_m,&\qqdn
n=1+3m;
\\[4mm]
0,&\qqdn n=2+3m.
\end{cases}
\enm
\end{exam}

\begin{exam}[$\ell=2$ in Theorem \ref{thm-h}]
\bnm
 &&\!\xxqdn_3F_2\ffnk{cccc}{\frac{4}{3}}{-n,&\frac{1+3b}{2},&\frac{4+3b}{2}}{&2+3b,&b-\frac{n-2}{3}}
 \\&&\!\xxqdn\:=\:
\begin{cases}
\fnk{ccccc}{\frac{1}{3},&\frac{2}{3}}{\frac{1}{3}-b,&\frac{2}{3}+b}_m,&\qqdn
n=3m;
\\[4mm]
\frac{-2}{3b+2}\fnk{ccccc}{\frac{2}{3},&\frac{4}{3}}{\frac{2}{3}-b,&\frac{4}{3}+b}_m,&\qqdn
n=1+3m;
\\[4mm]
\frac{2}{3(b+1)(3b+2)}\fnk{ccccc}{\frac{4}{3},&\frac{5}{3}}{1-b,&2+b}_m,&\qqdn
n=2+3m.
\end{cases}
\enm
\end{exam}

Although Theorems \ref{thm-f} and \ref{thm-h} can produce countless
$_3F_2(\frac{4}{3})$-series identities related to
\eqref{andrews-reversal} with the change of $\ell$, many known
results due to Chu and Chen \cite{kn:chen-a,kn:chen-b} can't be
covered by them.

\textbf{Acknowledgments}

The work is supported by the Natural Science Foundations of China
(Nos. 11301120, 11201241 and 11201291).
%%%%%%%%%%%%%%%%%%%%%%%%%%%%%%%%%%%%%%%%%%%%%%%%%%%%%%%%%%%%%%%%%%%
%%%%%%%%%%%%%%%%%%%%%%%%%%%%%%%%%%%%%%%%%%%%%%%%%%%%%%%%%%%%%%%%%%%
%%%%%%%%%%%%%%%%%%%%%%%%%%%%%%%%%%%%%%%%%%%%%%%%%%%%%%%%%%%%%%%%%%%

%%%%%%%%%%%%%%%%%%%%%%%%%%%%%%%%%%%%%%%%%%%%%%%%%%%%%%%%%%%%%%%%%%%
%%%%%%%%%%%%%%%%%%%%%%%%%%%%%%%%%%%%%%%%%%%%%%%%%%%%%%%%%%%%%%%%%%%
%%%%%%%%%%%%%%%%%%%%%%%%%%%%%%%%%%%%%%%%%%%%%%%%%%%%%%%%%%%%%%%%%%%

\end{document}